\documentclass[11pt]{amsart}
\usepackage[linktocpage]{hyperref}
\usepackage{amssymb, paralist, xspace, graphicx, url, amscd, euscript, mathrsfs,stmaryrd,epic,eepic,color}
\usepackage[all]{xy}
\usepackage{amsthm}
\usepackage{enumerate}
\usepackage{tikz}
\usepackage{tikz-cd}
\usetikzlibrary{calc}
\usetikzlibrary{arrows}
\usetikzlibrary{matrix}
\usepackage{lipsum}
\SelectTips{cm}{}

\numberwithin{equation}{section}

1
\numberwithin{subsection}{section}

\allowdisplaybreaks[1]

\newtheorem*{namedtheorem}{\theoremname}
\newcommand{\theoremname}{testing}

\newtheorem{theorem}[subsection]{Theorem}

\newtheorem{proposition-definition}[subsection]
{Proposition-Definition}
\newtheorem{corollary}[subsection]{Corollary}
\newtheorem{lemma}[subsection]{Lemma}
\theoremstyle{definition}

\newtheorem{remark}[subsection]{Remark}

\theoremstyle{remark}

\newcommand\cF{\mathcal{F}}

\newcommand\cH{\mathcal{H}}

\newcommand\cJ{\mathcal{J}}

\newcommand\cM{\mathcal{M}}

\newcommand\cO{\mathcal{O}}

\newcommand\cS{\mathcal{S}}

\newcommand\cU{\mathcal{U}}

\newcommand\cW{\mathcal{W}}

\newcommand\cZ{\mathcal{Z}}

\newcommand\CC{\mathbb{C}}

\newcommand\HH{\mathbb{H}}

\newcommand\PP{\mathbb{P}}

\newcommand\ZZ{\mathbb{Z}}

\newcommand\fR{\mathfrak{R}}
\newcommand\fS{\mathfrak{S}}

\newcommand\fX{\mathfrak{X}}

\newcommand\sse{\overline{{\rm Sec}}}
\newcommand\bse{{\rm Bse}}
\newcommand\se{{\rm Sec}}
\newcommand\Sch{ \rm{Sch}}
\newcommand\Set{{\rm Set}}

\newcommand\Pic{{\rm Pic}}

\newcommand{\talpha}{\tilde{\alpha}}

\theoremstyle{plain}
\theoremstyle{definition}

\begin{document}

\title{Integral Hodge classes on fourfolds fibered by quadric bundles}
\author{Zhiyuan Li,~~Zhiyu Tian}

\address{Department of Mathematics\\
building 380\\
Stanford, CA 94305\\
U.S.A.}

\address{Department of Mathematics\\California Institute of Technology\\ Pasadena, CA,
91125\\
U.S.A.}

\begin{abstract}
We discuss the space of sections and certain bisections on a quadric surfaces bundle $X$ over a smooth curve. The Abel-Jacobi from these spaces to the intermediate Jacobian will be shown to be  dominant with rationally connected fibers.  As an application, we prove that the integral Hodge conjecture holds for degree four integral Hodge classes of fourfolds fibered by quadric bundles over a smooth curve. This  gives an alternative proof of a result of Colliot-Th{\'e}l{\`e}ne and Voisin.
\end{abstract}
\maketitle

\section{Introduction}
Let $X$ be a smooth complex projective variety of dimension $n$. We denote by $Hdg^{2k}_\ZZ(X)=H^{k,k}(X)\cap H^{2k}(X,\ZZ)$ the group of integral Hodge classes of degree $2k$, and set 
$$Z^{2k}(X):=Hdg^{2k}_\ZZ(X)/\left<Z\subset X, Z\text{ has codimension k}\right>.$$
Then the Hodge conjecture  holds for degree $2k$ Hodge classes (resp. integral Hodge classes) if and only if the group $Z^{2k}(X)$ is torsion (resp. trivial).
In case that $k=2$ or rather for $X$ being a fourfold, the Hodge conjecture could be  closely related to the rational connectedness of $X$. For instance, it has been proved by Bloch and Srinivas \cite{BS83} that the Hodge conjecture holds for degree 4 Hodge classes on rationally connected varieties and also on  uniruled fourfolds using the diagonal decomposition. It remains  interesting to know whether Hodge conjecture holds for degree 4 integral Hodge classes (IHC) on such fourfolds.  

Recently, there are several methods to detect this problem. In \cite{CV12}, Colliot-Th{\'e}l{\`e}ne and Voisin have shown that $Z^4(X)=0$  when $X$ admits a quadric surfaces fibration over some surface via a K-theoretic approach. In this paper, we give an alternative proof of Hodge conjecture for IHC on such fourfolds via  a geometric approach developed by Voisin.  The main results are:

\begin{theorem}\label{thm:0} Let $f:X\rightarrow B$ be a family of quadric surfaces over $B$ with at worst quadric cone. Let $\se(X/B,h)$ (resp.~$\bse(X/B,h)$) be the space of sections (resp.~special bisections) (see $\S$4) of $f$. Then the morphism from $\bse(X/B,h)$ or $\se(X/B,h)$  to the torsor  $J(X)_h$ defined by Deligne-cycle  map is dominant and each fiber is rational. 
\end{theorem}

 \begin{theorem}\label{thm:1} Let $\PP(V)\rightarrow B$ be a $\PP^3$-bundle over a smooth projective curve $B$. 
Let $\pi:\fX\rightarrow T$ be a projective good family of quadric surfaces bundles in $\PP(V)$. Suppose that general fiber $\fX_t$ of $\pi$ satisfies the condition above. Then $Z^4(\fX)=0$. 
\end{theorem}

\begin{remark}The notion of a family being {\it good} is given in $\S$1. Roughly speaking, it means that the local system $R^4\pi_\ast\ZZ$ is trivial and all the fibers of the family have at worst ordinary double points as singularities.  
\end{remark}

\begin{remark}
	One can see that the fourfold $\fX$ is birational to a quadric surface bundle over a surface and hence  Theorem 1.2 can be implied by Colliot-Th{\'e}l{\`e}ne and Voisin's result. But we hope that the geometric method can be applied  to fourfolds admitting a toric surfaces fibration.
\end{remark}


\subsection{Organization of the paper} In section 2, we explain Voisin's method of  proving Hodge conjecture of integral Hodge classes on families of threefolds. The principal of her method is to find  sufficient many families of curves on a threefold $X$, which admit a rationally connected fibration onto its intermediate Jacobian $J(X)$ via Abel-Jacobi map. 

We review the classical geometry of quadric bundlesin  Section 3 and give some natural geometric modification of the quadric bundle.   In section 4, we show that the morphism induced by  Abel-Jacobi map agrees with the map defined by Hassett and Tschinkel \cite{HT12} (See also \cite{Zh12}).  Theorem \ref{thm:0} is proved in Theorem 4.4 and Theorem 4.13. The main theorem is proved in the last section.

\subsection{Acknowledgment} The authors benefits from the discussion with Brendan Hassett, Yi Zhu and Jun Li. The authors are also very grateful to Colliot-Th{\'e}l{\`e}ne for his helpful comments and explanation of the approach in \cite{CV12}. 

\section{Intermediate Jacobian and Abel-Jacobi map}

\subsection{Intermediate Jacobian}
Let $X$ be a smooth projective threefold. The intermediate Jacobian 
$$J(X):=H^3(X,\CC)/(F^2H^3(X)+H^3(X,\ZZ)),$$ is a compact torus. In particular, when $H^{3,0}(X)=0$ (e.g. when $X$ is rationally connected), the intermediate Jacobian $J(X)$ is an abelian variety.   It fits into the exact sequence $$0\rightarrow J(X)\rightarrow H_D^4(X,\ZZ(2))\rightarrow Hdg_\ZZ^4(X)\rightarrow 0,  $$
 where $H_D^4(X,\ZZ)=\HH^4(0\rightarrow\ZZ\rightarrow \cO_X\rightarrow \Omega_X\rightarrow 0)$ is the Deligne cohomology group.
 
For any integral Hodge class $\alpha\in Hdg^4_\ZZ(X)$, the torsor $J(X)_\alpha$ is defined to be the preimage of $\alpha$ in $H_D^4(X,\ZZ(2))$. Let $CH^2(X)_{\alpha}$ be the set of codimension two cycles in the Chow group whose cycle class is $\alpha$. Then there is a Deligne cycles class map
$$c_X:CH^2(X)_\alpha\rightarrow J(X)_\alpha\subseteq H^4_D(X,\ZZ(2)).$$
If $\alpha=0$, it can be identified with the Abel-Jacobi map $AJ_X:CH^2(X)_0\rightarrow J(X)$ introduced by Griffiths \cite{Gr69}. 

Let  $Z\subset M\times X$ be a family of 1-cycles of class $\alpha$, i.e. $[Z_b]=\alpha$ in $H^4(X,\ZZ)$ for all $b\in M$.  Then the Deligne cycle class map induces a morphism (cf.~\cite{Vo1})
\begin{equation}\label{eq2.1}
 \phi_Z:M\rightarrow J(X)_\alpha,
\end{equation}
defined by $\phi_Z(b)=c_\alpha(Z_b)$.

\begin{remark}\label{AJmap} Fix $b_0\in M$,  we can also define a map
\begin{equation}
\begin{aligned}
 \phi^{b_0}_Z:M&\rightarrow J(X)\\
 b&\mapsto AJ_X([Z_b]-[Z_{b_0}]),            
\end{aligned}
\end{equation}
which has the same fiber as $\phi_Z$. For simplicity,  we continue to use $\phi_Z$ to denote this morphism.
\end{remark}
 
These constructions naturally extend to the relative situation. Namely, let $f:\fX\rightarrow T$ be a family of rationally connected threefolds over a smooth curve $T$,  and assume that $f$ is smooth over $T_0\subset T$. For any section $\tilde{\alpha}$ of the local system $R^4f_\ast \ZZ$,  we get a family of torsors 
$$\cJ_{\tilde{\alpha}}\rightarrow T_0,$$ whose fiber over $t\in T_0$ is $J(\fX_t)_{\tilde{\alpha}(t)}$. 
Given a variety $\cM$ over $T$ and a family of relative 1-cycles $\cZ \subset \cM\times_T \fX$ of class  $\tilde{\alpha}$,  i.e. $[\cZ_t]=\tilde{\alpha}(t)$,  the restriction to $T_0$ gives a morphism 
\begin{equation}
\label{eq2.2} \Phi_\cZ:\xymatrix{ \cM\times_T T_0\ar[r]^{~}  \ar[d]^~ &\ \cJ_{\tilde{\alpha}}\ar[dl]^{~}\\ T_0  &~  } 
 \end{equation}
 In this paper, we are interested in the geometry of the map \eqref{eq2.2}. 
 
\subsection{Voisin's criterion}  
As discussed in \cite{Vo13},  the rational connectedness of the general fiber of \eqref{eq2.2} is closely related to integral Hodge conjectures on $\fX$.

Before we proceed,  let us first make some assumptions on our family $\fX$. We say that $\fX\rightarrow T$ is a {\it good} family if it satisfies the following conditions:
\begin{enumerate}
  \item $R^4f_\ast \ZZ$ is trivial;
  \item  $H^2 (\fX_t, \cO_{\fX_t})=H^3(\fX_t,\cO_{\fX_t})=0$; and $H^3(\fX_t,\ZZ)$ is torsion free for any smooth fiber $\fX_t$;
  \item Every fiber has at worst one ordinary double point as singularities.
\end{enumerate}
\begin{remark}
In the case of Theorem \ref{thm:1},  we know that  $H^\ast(\fX_t, \cO_{\fX_t})=0$ because of the uniruledness.  Moreover, since the integral cohomology group of $\PP(V)$ is torsion free, 
 $H^3(\fX_t,\ZZ)$ is automatically torsion free by Lefschetz hyperplane theorem and universal coefficients theorem.
\end{remark}

When $\fX\rightarrow T$ is a good family,  the following criterion is proved in \cite{Vo13}: 

\begin{theorem}\label{Voisincriterion}The group $Z^4(\fX)$ is trivial if for any section $\tilde{\alpha}$ of $R^4f_\ast \ZZ$, the following condition holds:

($\ast$) There exists a variety $g_{\tilde{\alpha}}: \cM_{\tilde{\alpha}}\rightarrow T_0$ and a family of relative  1-cycle $\cZ_{\tilde{\alpha}} \subset \cM_\alpha\times_T \fX$ of class $\tilde{\alpha}$, such that the morphism $\Phi_{\cZ_{\tilde{\alpha}}}:\cM_{\tilde{\alpha}}\rightarrow \cJ_{\tilde{\alpha}}$ is surjective with rationally connected general fibers. 
\end{theorem}

Moreover, every algebraic cycle $W\in CH^2(\fX)$ will induce a section $[W]$ of $R^4f_\ast \ZZ$. Suppose that condition $(\ast)$ holds for some section $\tilde{\alpha}$, then it holds for  the section $\tilde{\alpha}':=\pm\tilde{\alpha}+n[W]$ for all $n\in \ZZ$. This is because one can just take 
$$\cM_{\tilde{\alpha}'}=\cM_{\tilde{\alpha}}  ~and ~\cZ_{\tilde{\alpha}'}=\pm\cZ_{\tilde{\alpha}}+n(\cM_{\tilde{\alpha}'}\times_T W),$$
which will naturally satisfy the condition $(\ast)$. Therefore, we can obtain  the following result:  
  
\begin{corollary}\label{simplecriterion}
Let  $A^4(\fX)\subseteq Hdg_\ZZ^4(\fX)$ be the image of cup product $$\Xi:\Pic(\fX)\times \Pic(\fX)\rightarrow Hdg_\ZZ^4(\fX).$$ 
Then $Z^4(\fX)=0$ if the condition $(\ast)$ holds for sections of $R^4f_\ast\ZZ$ modulo the sections induced from $A^4(\fX)$.
\end{corollary}

\section{Classical geometry on quadric surfaces bundles}
Assume that $\pi:X\rightarrow B$ is  quadric bundle whose singular fibers are  at worst  a quadric cones. 
For simplicity, we assume that $\pi:X\rightarrow \PP^1$ has non trivial monodromy. Then the cohomology group $H^4(X,\ZZ)\cong \ZZ^2$ has rank two.  

Denote by $\Delta \subset \PP^1$  the set of points where the fiber is singular and set $m=|\Delta|$ to be  the number of singular fibers. Then $m$ is an even integer by \cite{HT84}.
 The  double cover $$g:C\rightarrow B$$ ramified along $\Delta$ is called the {\it discriminant curve} of the family $X\rightarrow B$, endowed with an involution $\iota:C\rightarrow C$. 

\subsection{Fano scheme of lines}  Let $\cF$ be the space of lines in the fibers of $\pi$ and $\cU\subset \cF\times X$ the universal family.  It is well-known that $$\cF\rightarrow C$$ is a smooth $\PP^1$-bundle obtained from the Stein factorization of $\cF\rightarrow \PP^1$.  

\begin{remark}
If the morphism $\pi:X\rightarrow B$ is smooth, then $g:C\rightarrow B$ is an \'etale covering of degree two. Moreover, in the case $B=\PP^1$, $\pi$ is smooth if and only the fibration $X\rightarrow \PP^1$  has trivial monodromy.\end{remark}

Given a section $\delta:C\rightarrow \cF$, the restriction $F_\delta\subseteq \cU$ of the universal family  to the image $\delta(C)\cong C$ gives a one dimensional family of lines on $X$, that is, a diagram: 
\begin{equation}\label{cylinder}
\xymatrix{F_{\delta}  \ar[d]^{q_\delta }\ar[r]^{p_\delta} &  X\\  C&}.
\end{equation}
This yields a cylinder homomorphism 
\begin{equation}
\begin{aligned}
\Psi_\delta: H_1(C,\ZZ)&\rightarrow H_3(X,\ZZ)\\
\gamma &\mapsto p_\delta (q_\delta^{-1}(\gamma)),
\end{aligned}
\end{equation}
which is an isomorphism by \cite{T72} Lecture 5. Then we get an isomorphism $$J(X)\xrightarrow{\sim}J(C)$$ via this cylinder homomorphism. Moreover, using the identification of intermediate Jacobians or Jacobians via cycle class map, we can rewrite the above isomorphism in terms of cycles:
\begin{lemma}\label{isomorphism1}
The map
\begin{equation} \label{eq3.2}
\psi_\delta:J(C)\xrightarrow{\sim} J(X).
\end{equation}
defined by $\psi_\delta(\sum\limits_{i} n_i [c_i])=c_X(\sum\limits_{i} n_i [p_\delta (q^\ast_\delta(c_i))])$ is an isomorphism, where $c_i$ are points in $C$ and $\sum\limits_i n_i=0$.
\end{lemma}

We may omit the notation $c_X$ in the latter sections for simplicity.

\subsection{Geometric modification}  The ruled surface $F_\delta$ obtained  in \eqref{cylinder} is birational to the Fano scheme of lines $\cF$. To describe this birational map, 
let us recall a useful modification of the  family $X\rightarrow B$ in \cite{HT12}. 

Let $\widetilde{Y}$ be the blow-up of $X\times_{B}C$ along the nodes of the singular fibers and let $\mu: \widetilde{Y}\rightarrow Y$ be the blowing-down of the strict transform of the singular fibers of $X\times_{B} C$.  Then $Y\rightarrow C$ is a smooth family of quadric surfaces with trivial monodromy. 

Moreover, the universal family $\cU$ is actually the  small resolution of $X\times_{B} C$, which blows up $X\times_{B} C$ along a family of lines on $X\times_{B}C$. Hence we have  a commutative diagram:
\begin{equation}\label{birational}
\xymatrix{ &\cU\ar[dl]_{q}\ar[d]^{p} &\widetilde{Y}\ar[dl]_{\eta}\ar[d]\ar[l]_{\mu'} \ar[r]^{\mu}&Y\\ \cF \ar[d]_{r} & X\ar[d]^{\pi} & X\times_{B}\ar[l] C\ar[d] &\\  C\ar[r]^{g} & B& C\ar[l] &}
\end{equation}
where $\mu': \tilde{Y}\rightarrow \cU $ is the blowing down the  one of  the rulings of the exceptional divisors over the nodes.

\begin{lemma}\cite{HVV11}
The Fano scheme of lines of the fibers of $Y\rightarrow C$ is a disjoint union of two ruled surfaces $\cF\cup \iota^\ast\cF$.  Each of them parametrizes one of the two rulings on the fibers of $Y\rightarrow C$ and the corresponding  universal family $\cW$ (or $\iota^\ast\cW$) is isomorphic to $Y$.
\end{lemma}

With the notation as in $\S$3.1, we let $R_\delta:=\cW|_{\delta(C)}$ be the restriction of $\cW$ to the curve $\delta(C)\subseteq \cF$. Then $R_\delta$ is a family of one of the two rulings on the fibers of $Y$. Note that the ruled surface $R_{\delta}$ is actually isomorphic to $\cF$ in this case. We thus obtain birational maps from \eqref{birational} as follows:
\begin{equation}\label{birationalofsurface}
\begin{tikzcd}
\cW\cong Y\arrow[dashed]{r}
  & \cU\arrow{r} & X\\
 \cF\cong R_\delta \arrow[hookrightarrow]{u} \arrow[dashed]{r}{\Lambda_\delta} 
  & F_\delta \arrow[hookrightarrow]{u} &
\end{tikzcd}
\end{equation}

Geometrically, for $z\in \cF$, if we denote by $\ell(z)$ the corresponding lines on $X$,  the birational map $\Lambda_\delta:\cF\dashrightarrow \cF_\delta$ is defined away from the fibers over points in $\Delta$ and is given by 
\begin{equation}
\Lambda_\delta(z)= \ell(\iota(z))\cap \ell(\delta\circ r(z)) \in \ell(\delta\circ r(z)),
\end{equation}
for all $z\in \cF\backslash r^{-1}(\Delta)$.

\begin{remark}\label{jacobiansmfibration}
Similarly as Lemma \ref{isomorphism1}, we have an isomorphism 
\begin{equation}\psi'_\delta: J(C)\xrightarrow{\sim} J(Y)
\end{equation}
from the cylinder homomorphism $\Psi_\delta': H_1(C,\ZZ)\rightarrow H_3(Y,\ZZ).$
\end{remark}

\subsection{Weil restriction} A beautiful geometry fact  is that the quadric fibration $X\rightarrow B$ can be reconstructed from the ruled surface $\cF\rightarrow C$ by Weil restriction. More precisely,  one can define 
a contravariant functor 
\begin{equation}\begin{aligned}
\fR_{C/B} \cF: (\cS ch_{B})^\circ &\longrightarrow~~~ (\rm{Sets})\\
T&~\longrightarrow {\rm Hom}_C(T\times_{B} C, \cF)
\end{aligned}\end{equation}  It known (cf.~\cite{BLR90}) that  the functor $\fR_{C/B}\cF$ is representable by a scheme $\rm{Res}_{C/B} \cF$ over $B$  and there is a functorial isomorphism
\begin{equation}
\label{weilid}\rm{Hom}_{B}(V, \rm{Res}_{C/B} \cF)\xrightarrow{\sim} \rm{Hom}_C(V\times_{B} C, \cF),
\end{equation}
of functors in $V$,  where $V$ varies over all $B$-schemes.
 
As shown in \cite{HT12}, Hasset and Tschinkel indicate the following diagram:

\begin{equation}\label{weilbirational}
\xymatrix{ &\widehat{X}\ar[dl]_{\beta} \ar[dr]^{\gamma}& \\ \rm{Res}_{C/B}\cF \ar[dr]^{\varpi} & &\hspace{0.8cm}X\hspace{0.8cm}\ar[dl]_{\pi}\\   & B&},
\end{equation}
where the arrows are described as follows:
\begin{enumerate}
\item $\varpi^{-1}(b)=\rm{Sym}^2 (r^{-1}(b))$ is isomorphic to $\PP^2$ over $b\in \Sigma$; 
\item $\beta$ is the blowing up of $ \rm{Res}_{C/B}\cF$ along  the diagonal in $\varpi^{-1}(b)$ over each point $b\in \Sigma$; 
\item $\gamma$ is the blowing  down of $\widehat{X}$ along  the proper transform of $\varpi^{-1}(b)$ in $\widehat{X}$ over each point $b\in \Sigma$.
\end{enumerate}

In particular, if $X\rightarrow B$ is a smooth quadric bundle, the Weil restriction  $\rm{Res}_{C/B}\cF$ is isomorphic to $X$.

\section{Sections on Quadric fibrations}
With the same assumption as in $\S$3, we are going to find families of curves on $X$ satisfying the conditions  in Theorem \ref{Voisincriterion}. The natural candidates are  families of sections and  bisections of $X\rightarrow B$.  In this section, we prove Theorem \ref{thm:0} and show that there exists  infinitely many families of sections and bisections satisfying $(\ast)$. Throughout this section, we assume that the family $X\rightarrow B$ has non-trivial mondromy. 

\subsection{Notation and Conventions}
We denote by $\cS ch_{S}$ or $\cS ch_{\CC}$ the category of schemes over a scheme $S$ or complex numbers $\CC$.

Let $\pi:X\rightarrow B$ be a projective family of varieties over a curve $B$.  The moduli functor  
$$\fS(X/B): (\cS ch_\CC)^{\circ} \rightarrow (\Set)$$ sends any $T\in\Sch_\CC $ to the set of families of sections of $X\rightarrow B$ over $T$. By \cite{Gr62} Part IV4.c,  $\fS(X/B)$ is representable by a scheme $\se(X/B)$, which is a union of countably many quasi-projective varieties.  
 
In this section,  for a cycle $z\in CH_1(X)$, we will use $[z]$ to denote the image of the cycle class map $c_X(z)$. Moreover,  if $\sigma:B\rightarrow X$ is a section of $\pi$, we continue to use $\sigma$ to represent the cycle class of  $\sigma(B)$ on $X$.
 
\subsection{Families of sections} 
Let us consider the space of sections on $X$.  For any non-negative integer $h \in \ZZ^{\geq 0}$, we define $$\se(X/B, h):=\{\sigma: B\rightarrow X|~ \sigma_\ast[B]=(1,h)\in H^4(X,\ZZ)\}$$ to be the space of smooth sections on $X\rightarrow B$ of class $(1,h)$.  Then $\se(X/B)$ is the union of all $\se(X/B,h)$ for  $h\in \ZZ^{\geq 0}$.

\begin{remark}
The space $\se(X/B, h)$  is  equivalent  to the space defined in \cite{HT12} using the {\it height} of sections. 
Here, the  height of a section $\sigma\in \se(X/B, h) $ is defined as $\deg(\sigma^\ast \omega_\pi^{-1})$, which only depends on $h$. 
\end{remark} 

The space $\se(X/B,h)$ is quasi-projective with natural compactifications. In this paper,  we regard the $\se(X/B, h)$ as an open subset of the Hilbert scheme of $X$ and denote by
  $\sse(X/B,h)$ the closure of $\se(X/B, h)$ in the Hilbert scheme parameterizing sections in $\cF$. 
  
Recall that we have a morphism induced by the Deligne-cycle class map
\begin{equation}
\phi_h:\se(X/B, h)\rightarrow J(X)_h
\end{equation}
with $\phi_h(\sigma)=[\sigma]$,  which natural extends to $\sse(X/B,h)$. Our first main result  is: 

\begin{theorem}\label{RCsec}
For $h>>0$, the morphism $\phi_h: \se(X/B,h)\rightarrow J(X)_h$  is the composition of an open subset and a projective bundle morphism. In particular,  the extended map $\bar{\phi}_h: \sse(X/B,h)\rightarrow J(X)_h$  is surjective with rationally connected general fibers.
\end{theorem}

The proof relies on the standard argument of  the ``reduction to the discriminant argument" (cf.~\cite{HT12} $\S$3).  We now review this reduction and  divide the proof into two steps: 
\\
 
\noindent {\bf Step 1}. 
Denote by $\se(\cF/C)$ the space of sections on the ruled surface $\cF\rightarrow C$,  then there is a  natural one-to-one map
\begin{equation}\label{seciso}\theta: \se(X/B)\rightarrow\se(\cF/C),
\end{equation}
since there are unique two lines passing through a given point on a quadric surface.
 
\begin{lemma}\label{isomorphism2}
The map (\ref{eq3.2}) is an isomorphism.
\end{lemma}
\begin{proof} Given a section $\sigma:B\rightarrow X$, since the smooth fibers of $X\rightarrow B$ and $\rm{Res}_{C/B} \cF\rightarrow B $ are isomorphic and the ambient spaces are smooth,  the pullback $\gamma^\ast$ and composition $\beta_\ast$ of $\sigma$ remains a section $\beta_\ast\gamma^\ast(\sigma):B\rightarrow  \rm{Res}_{C/B}\cF$ of $\varpi$. By representability of moduli functors,  this actually gives a morphism 
\begin{equation}\label{firstmap}
\begin{aligned}\rho:\se(X/B) &\longrightarrow\se(\rm{Res}_{C/B}\cF/C)\\
\sigma&\longmapsto \beta_\ast\gamma^\ast(\sigma).
\end{aligned}
\end{equation}
induced from the natural transformation
$\fS(X/B)\Longrightarrow \fS(\rm{Res}_{C/B} \cF/B)$ using the base change $\gamma^\ast$ and composition $\beta_\ast$.
 
Similarly,  we construct a natural transformation between two moduli functors $ \fS(\rm{Res}_{C/B} \cF/B)$ and $ \fS(\cF/C)$ using the universal property of  Weil restriction.  
Namely, for a family of sections $$ g:\Sigma\rightarrow S\times \rm{Res}_{C/B} \cF,$$ over a scheme $S$,  we get a unique map 
\begin{equation}\label{wrf}\fR(g): \Sigma\times_{B} C\rightarrow S\times \cF
\end{equation} by the canonical isomorphism \eqref{weilid}. Moreover,  it is easy to see that \eqref{wrf} is a family sections of $\cF\rightarrow C$ over $S$. This construction defines  a natural transformation 
\begin{equation}\label{naturaltransformation}\vartheta:\fS(\rm{Res}_{C/B} \cF/B) \longrightarrow \fS(\cF/C). 
\end{equation}
Once again,  we obtain a morphism  $$\theta':\se(\rm{Res}_{C/B}\cF/B)\rightarrow \se(\cF/C),$$ from \eqref{naturaltransformation} 
because of the representability of  two functors  $\fS(\rm{Res}_{C/B}\cF/B)$ and $\fS(\cF/C)$.

Then the map $\theta=\theta'\circ\rho$ is the composition of $\theta'$ and $\rho$, and hence  is a morphism. Moreover, it is separated and bijective. As we work over the filed $\CC$ and $\se(\cF/C)$ is smooth,  then it is an isomorphism by Zariski's main theorem for quasi-finite morphisms. 
\end{proof}
\begin{remark}
One can easily see the result above holds for  threefolds with quadric fibration over any curve $B$.
\end{remark}

Consider a section $\delta:C\rightarrow \cF$ as a curve on $\cF$, then we define $$\se(\cF/C, d)=\{\delta\in \se(\cF/C)|~\delta^2=d\},$$
which is an irreducible component of $\se(\cF/C)$. We have $$ \se(X/B,h)\xrightarrow{\sim} \se(\cF/C,d)$$  via the isomorphism \eqref{eq3.2} for  $d$  satisfying   \begin{equation} d=(1,h)\cdot c_1(X) -2+\frac{m}{2}-2g(B).
\end{equation}

\noindent{\bf Step 2}.  Before we proceed,  for simplicity of notations, we use Remark \ref{AJmap} to  modify our map $\phi_h$ as follows:
\begin{equation}
\begin{aligned}
\phi_h:\se(X/B,h)&\rightarrow J(X)\\
    \sigma&\mapsto  [\sigma]-[\sigma_0]
\end{aligned}
\end{equation}
where we fix a section $\sigma_0\in \se(X/B,h) $. Then it is equivalent to show that our first assertion holds for this refined map.  

Furthermore,  for any $\delta_0\in \se(\cF/C)$, Hassett and Tschinkel  \cite{HT12} have defined a map 
\begin{equation}\label{eq3.4}
\begin{aligned}
\widehat{\phi}_d: \se(\cF/C, &~d)\rightarrow J(C)\\
~~\delta~~~~~&~~~~\longmapsto\delta^\ast ([\delta_0])- D,
\end{aligned}
\end{equation}
where $D\in \Pic(C)$ is a divisor  and $\deg(D)=\delta^\ast([\delta_0])$. They have shown that this map $\widehat{\phi}_d$ is the composition of an open immersion and a projective bundle map for $h$ sufficiently large. 

\begin{remark}
We also recommend readers to \cite{Zh12} for a more general construction for homogeneous fibrations.
\end{remark}
 
The map $\widehat{\phi}_d$ is not canonical and depends on the choice of $\delta_0$ and $D$. However, if we take $\delta_0\in \se(\cF/C, d)$ and $D=\delta_0^\ast ([\delta_0])$, we claim that \eqref{eq3.4} is the same as the morphism $\phi_h$ up to an isomorphism  $J(X)\cong J(C)$, that is,
\begin{lemma}\label{commutativity}
The diagram 
\begin{equation}\label{comdiag}
\xymatrix{ \se(\cF/C,d)  \ar[d]^{\theta}\ar[r]^{~~~~~~~~~~~~\widehat{\phi}_{d}} & J(C)\ar[d]^{\psi_{\delta_0}}\\ \se(X/B,h)\ar[r]^{~~~~~~\phi_h} & J(X) }  
\end{equation}
is commutative, where the isomorphisms $\theta$ and $\psi_{\delta_0}$ are given in Lemma \ref{isomorphism1} and Lemma \ref{isomorphism2}.
\end{lemma}
\begin{proof}

To prove the assertion, we first show that the diagram \eqref{comdiag} is commutative up to the translation of a two torsion element, i.e. there exist a two torsion element $x\in J(X)$ such that 
$$\widehat{\phi}_d\circ\psi_{\delta_0} (\delta)=\phi_h\circ\theta(\delta) +x, $$ for all $\delta\in \se(\cF/C,h)$. 

Using the identifications of the fibers of $F_{\delta_0}\rightarrow C$ and the fibers of  $\cU\rightarrow \cF$ via the embedding $F_{\delta_0}\hookrightarrow \cU$,  we have   
$$\psi_{\delta_0}\circ\widehat{\phi}_d(\delta)=[p_\ast ( q^\ast (\delta\cdot\delta_0- \delta_0^2))],$$ where $q^\ast (\delta\cdot\delta_0- \delta_0^2)) $ is  a union of lines in the fibers of $X$.

On the other hand,  set $\sigma=\theta(\delta)$ and $\sigma_0=\theta(\delta_0)$,  and then we have  
 $$\phi_h\circ\theta(\delta)=[\sigma]-[\sigma_0],$$
 by definition. Recall that we have a commutative diagram \eqref{birational}. And we take  $\bar{\delta}:=p^{-1}(\sigma(B))\subset \cU$
 to be the inverse image of  $\sigma(B)$. Then it is easy to see that $\bar{\delta}$ is also contained in $F_{\delta}$ and can be viewed as a section of the ruled surface $ F_{\delta}$ over $\delta(C)\cong C$.
 
Observe that there exist two sections $\bar{\delta}'\subseteq F_{\delta}$ and $\bar{\delta}_0'\subseteq F_{\delta_0}$  such that  
 \begin{equation}p(\bar{\delta}_0')=p(\bar{\delta}').
 \end{equation} 
 This is because for general $b\in B$,  two lines passing though the point $\delta(b)$  will meet the two lines passing through the point $\delta_0(b)$ at exactly two points.
 \begin{remark}\label{relationofsections}
We can also view $\bar{\delta}'_0$  (resp.~$\bar{\delta}'$) as the image of $\delta$ (resp.~$\delta_0$) via the birational map \eqref{birational}. Then  $\bar{\delta}'_0$ is rationally equivalent to $\bar{\delta}$ if and only if $\delta$ is rationally equivalent to $\delta_0$.
 \end{remark}

Furthermore,  we have the following relations among these sections in $CH^2(\cU)$.

\begin{enumerate}[(1)]
\item $\bar{\delta}_0\equiv_{rat} \bar{\delta}'_0+ (\bar{\delta}'_0 \cdot \bar{\delta}_0 -\bar{\delta}^2_0)f_0$ and $\bar{\delta}\equiv_{rat} \bar{\delta}'+(\bar{\delta}' \cdot \bar{\delta} -\bar{\delta}^2)f$ in the corresponding ruled surface $F_{\delta_0}$ (resp.~$F_{\delta}$) and hence in $\cU$.

Here, $(\bar{\delta}'_0 \cdot \bar{\delta}_0 -\bar{\delta}^2_0)f_0$ (resp.~$(\bar{\delta}' \cdot \bar{\delta} -\bar{\delta}^2)f$) denote the corresponding fiber classes in $F_{\delta_0}$ (resp. $F_\delta$) containing  $(\bar{\delta}'_0 \cdot \bar{\delta}_0 -\bar{\delta}^2_0)$ (resp.~$(\bar{\delta}' \cdot \bar{\delta} -\bar{\delta}^2)$). 
\\
   
\item $ (\bar{\delta}'_0\cdot \bar{\delta}_0)f_0=q^\ast ( \delta\cdot\delta_0)=(\bar{\delta}'\cdot\bar{\delta})f$ in $\cU$; 
\\

\item  $(\bar{\delta}^2)f_0 \equiv_{rat} q^\ast (\delta^2)$  and $(\bar{\delta}_0^2)f \equiv_{rat} q^\ast p_\ast(\delta^2_0)$.
\end{enumerate}
Here, $(1)$ and $(2)$ are straightforward, and $(3)$ comes from $(2)$ and Remark \ref{relationofsections}.
Putting these together, we get
\begin{equation}
\begin{aligned}
2\phi_h\circ\theta(\delta)&=[p_\ast (\bar{\delta}-\bar{\delta_0})]\\
&=[p_\ast(\bar{\delta}-\bar{\delta}')]+[p_\ast(\bar{\delta}'_0-\bar{\delta}_0)]\\
 &=[p_\ast(\bar{\delta}^2f-\bar{\delta}_0^2f_0)]\\  
 &=[p_\ast q^\ast(\delta^2-\delta_0^2)]  \\
&= 2p_\ast q^\ast (\delta\delta_0-\delta_0^2)=2\psi_{\delta_0}\circ\widehat{\phi}_d(\delta).
\end{aligned}
\end{equation}
Therefore,  $2(\phi_h\circ\theta-\psi_{\delta_0}\circ\widehat{\phi}_d)=0$ which implies that the diagram is commutative up to a 2-torsion element. 

Moreover, for $h>>0$,  since the pullback and pushforward preserves the rational equivalence relation, we have  $$\phi_h\circ\theta(\delta_0)=0=\psi_{\delta_0}\circ\widehat{\phi}_d(\delta_0).$$
It follows that only $\phi_h\circ\theta=\psi_{\delta_0}\circ\widehat{\phi}_d$ happens.
\end{proof}

\begin{remark}{(Families with trivial monodromy)}
If $X\rightarrow B$ is a smooth quadric bundle with trivial monodromy, then $H^4(X,\ZZ)\cong \ZZ^3$ and the intermediate Jacobian $J(X)$ is isomorphic to the Jacobian $J(B)$.  

Similarly, for a class  $\epsilon=(1,a,b)\in H^4(X,\ZZ)$, we can get the natural map $$\sse(X/B, \epsilon)\rightarrow J(X)$$
is surjective and rationally connected for large $(a,b)$
\end{remark}

\subsection{Bisections on quadric fibrations} Next, we will consider  families of {\it special} bisections on $\pi:X\rightarrow B$. With the same notations as before,  let $\Sigma=Sing(\pi^{-1} (\Delta))$ be the set of points where the morphism $\pi$ is not smooth. We denote by $m$ the number of points in $\Sigma$.

Let us denote by $\bse(X/B, h)$ the space of bisections of class $(2,h)$ on $X$ which are ramified over all the points in $\Sigma$. Note that such a bisection $\sigma$ corresponds to  a genus $g=\frac{m}{2}+2g(B)-1$ and bidegree $(2,h)$ curve passing through all points in $\Sigma$.  One can view $\bse(X/B, h)$  as an open subset of the Hibert scheme  which parameterize curves on $X$ of genus $g$, bidegree  $(2, h)$ and passing through all the points in $\Sigma$. We denote by $\overline{\bse}(X/B,h)$ the Zariski closure in this Hilbert scheme.
\begin{remark}
Note that if $B\cong \PP^1$, then $m$ is always greater than zero with the assumption of non-trivial monodromy. 
\end{remark}
Recall that we have a smooth family of quadric surfaces $Y\rightarrow C$ by \eqref{birational}. According to Remark \ref{jacobiansmfibration}, we have 
$$\psi'_\delta:J(C)\xrightarrow{\sim} J(Y)$$ via the cylinder morphism $\Psi'_\delta$. Similarly as the non-smooth case, we also have $\se(Y/C)\cong \se(\cF/C)$.

Let $\zeta\in H^4(Y,\ZZ)$ be an integral class and  denote by $\se(Y/C, \zeta)$ the space of sections of class $\zeta$, which is a component of $\se(\cF/C)$. By the same argument in the proof of Theorem \ref{RCsec},  we get  that the morphism 
\begin{equation}\label{bisectionajmap}
\phi_\zeta: \se(Y/ C,\zeta)\rightarrow J(Y)_\zeta
\end{equation} 
induced by the cycle class map is the composition of an open immersion and a projective bundle map when $\zeta\cdot ([-K_Y])>>0$.

Since  $\se(Y/C,\zeta)\cong \bse(X/B,h)$ via the pushforward $\eta_\ast$ and pullback $\mu^\ast$,  where $h$ is uniquely (up to the involution) determined by $\zeta$,  then we get
\begin{theorem} \label{RCbsec}
The morphism $\phi_{(2,h)}: \bse(X/B,h)\rightarrow J(X)_{(2,h)}$ defined by Deligne-cycle class map is dominant and each fiber is rational. The extended map
 $$\overline{\phi}_{(2,h)}: \overline{\bse}(X/B,h)\rightarrow J(X)_{(2,h)}$$ is surjective with rationally connected general fibers.
\end{theorem}
\begin{proof} Still, we only need to show the first assertion. This comes from the commutativity of the following diagram
$$\begin{tikzpicture}
  \matrix (m) [matrix of math nodes,row sep=3em,column sep=4em,minimum width=2em]
  {
     \se(Y/C,\zeta) & J(Y)_\zeta\\
     \bse(X/B, h)& J(X)_{(2,h)} \\};
  \path[-stealth]
    (m-1-1) edge node [right] {$\simeq$} (m-2-1)
            edge node [above] {$\phi_\zeta$} (m-1-2)
    (m-2-1.east|-m-2-2) edge node [below] {~}
            node [above] {$\phi_{(2,h)}$} (m-2-2)
    (m-1-2) edge node [right] {$\simeq$} (m-2-2);
\end{tikzpicture}$$
where the vertical arrows are  induced from the pushforwad $\eta_\ast$ and the pullback  $\mu^\ast$.  Here, the arrow $J(Y)_{\zeta}\rightarrow J(X)_{(2,h)}$ is an isomorphism because we have isomorphisms of homology groups as below:
$$\begin{tikzpicture}
  \matrix (m) [matrix of math nodes,row sep=3em,column sep=4em,minimum width=2em]
  {
     H_3(R_\delta,\ZZ)  & H_3(Y,\ZZ) \\
     H_3(F_\delta, h)& H_3(X,\ZZ)  \\};
  \path[-stealth]
    (m-1-1) edge node [right] {$\simeq$} (m-2-1)
            edge node [above] {$\Psi'_\delta$} (m-1-2)
    (m-2-1.east|-m-2-2) edge node [below] {~}
            node [above] {$\Psi_{\delta}$} (m-2-2)
    (m-1-2) edge node [right] {$\eta_\ast\circ\mu^\ast$} (m-2-2);
\end{tikzpicture}$$
from  \eqref{birationalofsurface}.
 
\end{proof}

Then Theorem \ref{thm:0} follows from Theorem \ref{RCsec} and Theorem \ref{RCbsec}. 

\section{Proof of Theorem \ref{thm:1}}

\noindent{\it Proof of Theorem \ref{thm:1}}.  We  assume that the general fiber of $\pi$ admits a quadric fibration with non-trivial monodromy. Suppose $\pi:\fX\rightarrow T$ is smooth over the open subset $T_0\subseteq T$. Then the section $\talpha$ of $R^4\pi_\ast\ZZ$ is trivial over $T_0$,  and is determined by the cohomology class of bidegree $(d_1,d_2)$ on the smooth fiber $\fX_t, t\in T_0$. 

There exists two divisors $\cH_1, \cH_2$ on $\fX$ such that the restriction of $\cH_i, i=1,2$ to the general fiber is $i_{\fX_t}^\ast \cO_{\PP(V)} (1)$  and a fiber class. Then the restriction of  the cycles $\cH^2_1 $ and $\cH_1\cdot \cH_2$ induces two sections of $R^4\pi_\ast \ZZ$ of class  $(k,2)$ and $(2,0)$ for some $k>0$. As discussed in Corollary \ref{simplecriterion},  one only need to check that the condition ($\ast$) holds for  sections $\talpha$ of $R^4\pi_\ast \ZZ$ modulo group spanned by $[H_1^2]$ and $[H_1H_2]$. Since $[H_1^2]$ and $ [H_1H_2]$ are of class $(2,d)$ and $(2,0)$, this allow us to assume that $\talpha$ is of class $(1,h)$ or $(2,h)$ for some  $h>>0$.

Let us first consider the case where the section $\talpha$ of class $(1,h)$. Take $\cM_{\talpha}$ to be a desingularization of the relative Hilbert scheme of rational curves of bidegree $(1,h)$ in $\fX|_{T_0}$ and  let $\cZ_{\alpha}\subseteq \cM_{\talpha}\times_T \fX$ be the pullback of the universal family.   By Theorem \ref{RCsec}, the map
$$\phi_{\cZ_{\talpha}}:\cM_{\talpha}\rightarrow \cJ_{\talpha}(\fX),$$
is surjective, and  general fibers are rationally connected.
 
 Similarly, for a section $\talpha$ of bidegree $(2,h)$, we choose $\cM_{\talpha}$ to  be a desingularization of the relative Hilbert scheme 
 of genus $\frac{m-2}{2}$ curves in $\fX|_{T_0}$  of bidegree $(2,h)$, passing through all the vertices of the quadric cones in $\fX_t$. By Theorem \ref{RCbsec}, if we take  $\cZ_{\talpha}$ to be the pullback of the universal family,  the induced morphism
 $$ \phi_{\cZ_{\talpha}}:\cM_{\talpha}\rightarrow \cJ_{\talpha}(\fX),$$
is surjective with rationally connected general fibers. This completes the proof.
\qed



\bibliographystyle {plain}
\bibliography{IHC}

\end{document}